\documentclass[12pt]{article}

\usepackage{amsmath,amssymb,amsthm}
\newtheorem{theorem}{Theorem}

\newtheorem{corollary}[theorem]{Corollary}

\def\0{\leqno}
\title{Finite groups with a certain number\\ of cyclic subgroups}
\author{Marius T\u arn\u auceanu}
\date{February 17, 2015}

\begin{document}
\maketitle

\begin{abstract}
    In this short note, we describe the finite groups $G$ having $|G|-1$ cyclic subgroups.
    This leads to a nice characterization of the sy\-mme\-tric group $S_3$.
\end{abstract}

In subgroup lattice theory, it is a usual technique to associate
to a finite group $G$ some posets of subgroups of $G$ (see e.g.
\cite{4}). One such poset is the poset of cyclic subgroups of $G$,
usually denoted by $C(G)$. Notice that there are few papers
on the connections between $|C(G)|$ and $|G|$ (\cite{2,3} are examples). We
also recall the following basic result of group theory.

\begin{theorem}
    A finite group $G$ is an elementary abelian $2$-group if and only if $|C(G)|=|G|$.
\end{theorem}

Inspired by Theorem 1, we study here the finite groups
$G$ for which $$|C(G)|=|G|-1.\0(*)$$First, we observe that
certain finite groups of small orders, such as $\mathbb{Z}_3$,
$\mathbb{Z}_4$, $S_3$ and $D_8$, have this property. Our main
theorem proves that in fact these groups exhaust all finite groups
$G$ satisfying $(*)$.

\begin{theorem}
    Let $G$ be a finite group. Then $|C(G)|=|G|-1$ if and only if
    $G$ is one of the following groups: $\mathbb{Z}_3$,
    $\mathbb{Z}_4$, $S_3$ or $D_8$.
\end{theorem}

\begin{proof}
Assume that $G$ satisfies $(*)$, let $n=|G|$ and denote by
$d_1=1,d_2,...,d_k$ the positive divisors of $n$. If $n_i=|\{H\in C(G)\mid
|H|=d_i\}|$, $i=1,2,...,k$, then
$$\sum_{i=1}^k n_i\phi(d_i)=n.$$Since $|C(G)|=\sum_{i=1}^k
n_i=n-1$, one obtains
$$\sum_{i=1}^k n_i(\phi(d_i)-1)=1,$$which implies that:
\begin{itemize}
\item[-] there exists $i_0\in\{1,2,...,k\}$ such that $n_{i_0}=1$ and $\phi(d_{i_0})=2$ (i.e. $d_{i_0}\in\{3,4,6\}$);
\item[-] for an $i\neq i_0$, we have either $n_i=0$ or $\phi(d_i)=1$ (i.e. $d_i\in\{1,2\}$).
\end{itemize}

We remark that $d_{i_0}$ cannot be equal to 6 because in this case $G$ would also have a cyclic subgroup of order 3, a contradiction. We infer that $G$ contains a unique normal cyclic subgroup of order $d_{i_0}$, say $H$. Let $X$ be a subgroup of $G$ of prime order $p$. Then either $p=2$ or $p$ divides $d_{i_0}$. Hence by Cauchy's Theorem, either $d_{i_0}=4$ and $G$ is a $2$-group, or $d_{i_0}=3$ and $G$ is a $\{2,3\}$-group.

If $d_{i_0}=3$ it follows that $H$ is the unique Sylow $3$-subgroup of $G$, and consequently $G=HK$ for some $K\in Syl_2(G)$. The theorem holds if $G=H$, so we may assume $K\neq 1$. As $G$ has no cyclic subgroup of order $6$, we have  $C_K(H)=1$, so $|K|=|Aut(H)|=2$. Therefore $G$ is the nonabelian group $S_3$ of order $6$.

Assume next that $d_{i_0}=4$, so that $G$ is a $2$-group. If $G=H$, then the theorem holds, so
we may assume that there is $g\in G\setminus H$. Then, as each member of $G\setminus H$ is
an involution, $g$ is an involution inverting $H$ via conjugation. Since $|G:C_G(H)|\leq |Aut(H)|=2$, we
conclude that $G=H\langle g\rangle$ is the dihedral group of order 8. This completes the proof.
\end{proof}

The following corollary is an immediate consequence of Theorem 2.

\begin{corollary}
    $S_3$ is the unique finite group $G$ which is not a $p$-group
    and satisfies $|C(G)|=|G|-1$.
\end{corollary}

We also remark that a facile proof of Theorem 1 easily follows from the first part of the proof of Theorem 2. Finally, we indicate a natural open problem concerning the above study.

\bigskip\noindent{\bf Open problem.} Describe the finite groups $G$ satisfying
$|C(G)|=|G|-r$, where $2\leq r\leq |G|-1$.
\bigskip

\bigskip\noindent{\bf Acknowledgements.} The author is grateful to the reviewer for
its remarks which improve the previous version of the paper.

\vspace*{5ex}\small

\hfill
\begin{minipage}[t]{5cm}
Marius T\u arn\u auceanu \\
Faculty of  Mathematics \\
``Al.I. Cuza'' University \\
Ia\c si, Romania \\
e-mail: {\tt tarnauc@uaic.ro}
\end{minipage}


\begin{thebibliography}{10}
\bibitem{1} I.M. Isaacs, {\it Finite group theory}, Amer. Math. Soc., Providence, R.I., 2008.
\bibitem{2} G.A. Miller, {\it On the number of cyclic subgroups of a group}, Proc. N.A.S. {\bf 15} (1929), 728-731.
\bibitem{3} I.M. Richards, {\it A remark on the number of cyclic subgroups of a finite group}, Amer. Math. Monthly {\bf 91} (1984), 571-572.
\bibitem{4} R. Schmidt, {\it Subgroup lattices of groups}, de Gruyter Expositions in Ma\-the\-ma\-tics 14, de Gruyter, Berlin, 1994.
\end{thebibliography}
\end{document}